\documentclass[12pt]{article}
\usepackage{amsmath,amsthm,amssymb}
\allowdisplaybreaks
\usepackage{color}

\newcommand{\pa}{\partial}

\newcommand{\la}{\lambda}

\newcommand{\va}{\varepsilon}

\newcommand{\dis}{\displaystyle}

\theoremstyle{plain}
\newtheorem{theorem}{Theorem}[section]

\newtheorem{lemma}{Lemma}[section]

\newtheorem{corollary}{Corollary}[section]
\numberwithin{equation}{section}
\newcommand{\n}{{\mathcal  N}}
\newcommand{\R}{{\mathbb R}}
\newcommand{\N}{{\mathbb N}}

\begin{document}
	\title{A Global Compact Result for a Fractional Elliptic Problem with  Critical Sobolev-Hardy 
		Nonlinearities on $\R^N$
		\thanks {The research was supported by the Natural Science Foundation of China
			(11101160,11271141) and the China Scholarship Council (201508440330)  }}
	\author{ Lingyu Jin, Shaomei Fang\\\small{College of Science, South China Agriculture
			University,}
		\\\small{ Guangzhou 510642, P. R. China}
		\\
	}
	\date{}
	\maketitle
	\begin{abstract}
		In this paper, we are
		concerned with the following type of elliptic problems:
		$$
		\begin{cases}
		(-\Delta)^{\alpha} u+a(x) u=\frac{|u|^{2^*_{s}-2}u}{|x|^s}+k(x)|u|^{q-2}u,\ \
		\,\\
		u\,\in\,H^\alpha(\R^N),
		\end{cases}
		\eqno {(*)} $$
		where   $2<q< 2^*$, $0<\alpha<1$, $0<s<2\alpha$, $2^*_{s}=2(N-s)/(N-2\alpha)$ is the critical Sobolev-Hardy exponent,
		$2^*=2N/(N-2\alpha)$ is the critical Sobolev exponent, $a(x),k(x)\in C(\R^N)$. Through a compactness
		analysis of the functional associated to $(*)$, we
		obtain the existence of positive solutions for $(*)$ under
		certain assumptions on $a(x),k(x)$.
	\end{abstract}
	
	\bigbreak
	{\bf Key words and phrases}. \ \ fractional Laplacian,  compactness, positive solution, unbounded domain, Sobolev-Hardy nonlinearity.
	\bigbreak
	{AMS Classification:} 35J10 35J20 35J60
	\bigbreak

	\section{Introduction}
	We consider the following nonlinear elliptic equations:
	
	\begin{equation}\label{1.1}
	\begin{cases}
	(-\Delta)^{\alpha} u+a(x) u=\dis\frac{|u|^{2^*_{s}-2}u}{|x|^s}+k(x)|u|^{q-2}u, \,x\in \R^N\ \
	\,\\
	u\,\in\,H^\alpha(\R^N),
	\end{cases}
	\end{equation}
	where  $2<q < 2^*$, $0<\alpha<1$, $0<s<\alpha$, $2^*_{s}=2(N-s)/(N-2\alpha)$ is the critical Sobolev exponent, $2^*=2N/(N-2\alpha)$ is the critical Sobolev exponent,  $a(x),k(x)\in C(\R^N)$.
	
	In the case $\alpha=1$, problem (\ref{1.1}) with the Sobolev-Hardy term has been extensively studied  (see  \cite{CP}, \cite{CP1}, \cite{Ch}, \cite{DJP}, \cite{J}).
	For $0<\alpha<1$ the nonlocal operator $(-\Delta)^\alpha$ in $\R^N$ is defined on the Schwarz class through the Fourier form or via the Riesz potential. Recently the fractional and more general non-local operators of elliptic type have been widely studied, both for their interesting theoretical structure and concrete applications in many fields such as optimization, finance, phase transitions, stratified materials, anomalous diffusion and so on (see \cite{FB}, \cite{FQ}, \cite{NPV}, \cite{SV}, \cite{SV1}, \cite{SZY}). When $s=0$, (1.1) are the
	elliptic equation involving the nonlocal operator and the critical Sobolev nonlinearity.  Abundant  results  have been accumulated (see \cite{FB}, \cite{FQ}, \cite{SV}, \cite{SV1}, \cite{SZY}). When $0<s<2\alpha$, (\ref{1.1})  has the Sobolev-Hardy nonlinearity. In particular, recently Yang etc. in \cite{YY}, \cite{WY} considered  the existence of solutions for (\ref{1.1}) ($0<s<2\alpha$ or $s=2\alpha$) in a bounded domain. Motivated by it,  we consider the compactness analysis and thereby obtain the existence of the solutions for problem (\ref{1.1}) in $\R^N$. Compare with Yang's work, the new difficulty of this problem  that emerges here is the lack of compactness caused by the unbounded domain $\R^N$. As is well known, the translation invariance of $\R^N$ and the scaling invariance of critical exponents are typical
	difficulties in the study of elliptic equations. Indeed,
	such invariance disables  the compactness of the embeddings. To overcome the difficulties caused by the lack of compactness, we carry out a non-compactness analysis which can distinctly express all the parts which cause non-compactness. As  a result, we are able to obtain the existence of nontrival solutions of the
	elliptic problem including the critical nonlinear term on unbounded domain  by getting rid of
	these noncompact factors.
	To be more specific, for
	the Palais-Smale sequences of the variational functional
	corresponding to (\ref {1.1}) we first establish a  complete noncompact expression  which includ all the blowing up
	bubbles caused by critical Sobolev-Hardy and unbounded
	domains. Then by applying the noncompact expression, we derive the existence  of positive solutions for  (\ref {1.1}).  Our methods base on   some techniques of \cite{CP}, \cite {L1}, \cite {L2}, \cite{PP}, \cite {Sm}, \cite {S2}, \cite {ZC},  \cite{Y}, \cite{Y1}.

	Before introducing our main results, we give  some
	notations and assumptions.
	
	{\bf Notations and assumptions:}
	
	Let $N\geq 1$, $u\in L^2(\R^N)$, let the Fourier transform of $u$ be 
	$$\widehat{u}(\xi)=\mathcal{F}(\xi):=\frac{1}{(2\pi)^{\frac{N}{2}}}\int_{\R^N}e^{-i\xi\cdot x}u_0(x)dx.$$
	For  $\alpha>0 $, the  Sobolev space $H^\alpha(\R^N)$ is defined as the completion of $C^\infty_0(\R^N)$ with the norm
	$$\|u\|_{H^\alpha(\R^N)}=\|\widehat{u}\|_{L^2(\R^N)}+\||\xi|^\alpha\widehat{u}\|_{L^2(\R^N)}.$$
	
	Let $\dot{H}^\alpha(\R^N)$ be the homogeneous version  as the completion of $C^\infty _0(\R^N)$ under the norm
	$$\|u\|_{\dot{H}^\alpha(\R^N)}=\||\xi|^\alpha\widehat{u}\|_{L^2(\R^N)}.$$
	We define the operator $(-\Delta)^\alpha u,\alpha\in\R$ by the Fourier transform $$\widehat{(-\Delta)^{\alpha} u}(\xi)=|\xi|^{2\alpha}\hat{u}(\xi),\ \ \forall u\in C^\infty _0(\R^N).$$
	Then we have $$\||\xi|^\alpha\widehat{u}\|^2_{L^2(\R^N)}=\int_{\R^N}|(-\Delta )^{\alpha/2}u|^2dx.$$
By the Parseval identity, we also have
$$\|u\|_{H^\alpha(\R^N)}^2=\|u\|^2_{L^2(\R^N)}+\int_{\R^N}\int_{ \R^N} \frac{|u(x)-u(y)|^2}{|x-y|^{N+2\alpha}}dxdy=\|u\|^2_{L^2(\R^N)}+\int_{\R^N}|(-\Delta )^{\alpha/2}u|^2dx.$$
 Denote $c$ and $C$ as
arbitrary constants. Let $B(x,r)$ denote a ball centered at $x$ with
radius $r$, $B(r)$ denote a ball centered at 0 with radius $r$ and  $B(x,r)^C=\R^N\setminus B(x,r)$. 

In this paper we assume that:

(a) $0\leq a(x)\in C(\R^N)$, $0\leq k(x)\in C(\R^N)$;

(b) $\displaystyle\lim_{|x|\rightarrow\infty} a(x)=\bar
a>0,\lim_{|x|\rightarrow\infty} k(x)=
\bar k>0,\,\,\,\inf_{x\in\R^N}a(x)={\hat
a}>0,\,\inf_{x\in\R^N}k(x)={\hat
k}>0.$

 In the
following, we assume that $a(x),k(x)$ always satisfy (a) and (b).
The energy functional associated with (1.1) is
$$I(u)=\frac{1}{2}\int_{\R^N}\Bigl (|(-\Delta)^{\alpha/2}u|^2
+a(x)u^2\Bigl )
dx-\dis\frac{1}{2^*_{s}}\int_{\R^N}\frac{|u|^{2^*_{s}}}{|x|^s}dx-\frac{1}{q}\int_{\R^N}k(x)|u|^q
dx,\, \ \  u\in H^{\alpha}(\R^N).$$ We next present some problems
associated to (1.1) as the follows.

\noindent The limit equation of (1.1) at infinity is
\begin{equation}\label{(2.1)}
\begin{cases}
(-\Delta)^{\alpha} u+\bar a u=\bar{k}|u|^{q-2}u,\\
u\in H^{\alpha}(\R^N),
\end{cases}
\end{equation}
and its corresponding variational functional is
$$I^\infty(u)=\frac{1}{2}\int_{\R^N}\Bigl (|(-\Delta)^{\alpha/2}u|^2+\bar a
|u|^2\Bigl
)dx-\frac{1}{q}\int_{\R^N}\bar k
|u|^qdx, \,\,\ \  u\in H^{\alpha}(\R^N).$$ The limit equation of (1.1)
related to the Sobolev-Hardy critical nonlinear term is
\begin{equation}\label{1.4}
\begin{cases}
(-\Delta)^{\alpha} u=\dis\frac{|u|^{2^*_{s}-2}u}{|x|^s},\\
u\in \dot{H}^\alpha(\R^N),
\end{cases}
\end{equation}
and the corresponding variational functional is
$$I_s(u)=\frac{1}{2}\int_{\R^N} |(-\Delta)^{\alpha/2}u|^2dx-\frac{1}{2^*_{s}}\int_{\R^N}\frac{|u|^{2^*_{s}}}{|x|^s}dx,\,\,\ \ \
u\in \dot{H}^\alpha(\R^N).$$ 


 In \cite{Y1} Chen and Yang proved that
 all the positive solutions of (\ref{1.4}) are of the form
 $U^\varepsilon(x):=\varepsilon^{\frac{2\alpha-N}{2}} U
 (x/\varepsilon)$,
 and $U(x)$ satisfies \begin{equation}\label{1.5}
 \frac{C_1}{1+|x|^{N-2\alpha}}\leq U(x)\leq \frac{C_2}{1+|x|^{N-2\alpha}},
 \end{equation}
 where $C_2>C_1>0$ are constants. 
 These solutions are also minimizers
 for the quotient
 $$
 S_{\alpha,s}=\inf_{u\in
 	\dot{H}^\alpha(\mathbb{R}^N) \backslash\{0\}}\frac{\dis\int_{\mathbb{R}^N}|(-\Delta)^{\alpha/2}
 	u|^2dx}{\Bigl(\dis\int_{\mathbb{R}^N}\frac{|u|^{2^*_{s}}}{|x|^s}dx\Bigl)^{2/2^*_{s}}},$$ which is associated with the fractional Sobolev-Hardy inequality 
 \[ \Bigl(\dis\int_{\mathbb{R}^N}\frac{|u|^{2^*_{s}}}{|x|^s}dx\Bigl)^{2/2^*_{s}}\leq S_{\alpha,s}^{-1}\dis\int_{\mathbb{R}^N}|(-\Delta)^{\alpha/2}
 u|^2dx. \]
Define
$$D_0=\int_{\mathbb{R}^N}\Bigl(\frac{1}{2}|(-\Delta)^{\alpha/2}
U|^2-\frac{1}{2^*_{s}}\frac{|U|^{2^*_{s}}}{|x|^s}\Bigl)dx=\frac{2\alpha-s}{2(N-s)}S_{\alpha,s}^{\frac{N-s}{2\alpha-s}},$$
$$\n=\{u\in H^{\alpha}(\R^N)\setminus
\{0\}\,\,|\,\,\int_{\R^N}\Bigl(|(-\Delta)^{\alpha/2}u|^2+\bar
a|u|^2-\bar k|u|^q\Bigl) dx=0\},$$
and 
$$J^\infty=\inf_{u\in \n}I^\infty (u).$$
It is known that $\n \neq\emptyset$ since problem (\ref{(2.1)}) has
at least one positive solution if $N> 2\alpha$ (see \cite{S})
for $1<q<2^*$.

The main result of our paper is as follows:
\begin{theorem}\label{a}
Suppose $a(x),\,k(x)$ satisfy (a) (b),  $ 2<q<2^*, N> 2\alpha$. Assume that $\{u_n\}$ is a positive
Palais-Smale sequence of I at level $d\geq 0$, then there exist
two sequences
$\{R^j_n\}\subset\R^+\,(1\leq j\leq l_1$) and $\{y^j_n\}\subset\R^N\,(1\leq j\leq
l_2), \,0\leq u\in H^{\alpha}(\R^N)$, and $ 0< u_j\in H^{\alpha}(\R^N) \,(1\leq j\leq
l_2), (l_1,l_2\in \N)$ such that up to a subsequence:

$\bullet\,\,I' (u)=0,\,\,{I^{\infty} }'(u_j)=0\,(1\leq j\leq l_2);$

$\bullet\,\,R^j_n\rightarrow 0\,\,(1\leq j\leq
l_1)\text{ as }n\rightarrow\infty;$


$\bullet\,\,|y^j_n|\rightarrow \infty\,(1\leq j\leq l_2)\text{ as
}n\rightarrow\infty ;$

$\bullet\,\,\displaystyle
d=I(u)+l_1D_0+\sum^{l_2}_{j=1}I^\infty(u_j);$

\begin{equation}\label{b1}\bullet\,\,
\|u_n-u-\sum^{l_1}_{j=1}U^{R^j_n}
-\sum^{l_2}_{j=1}u_j(x-y^j_n)\|_{H^{\alpha}(\R^N)}=o(1)\text{ as
}n\rightarrow\infty.\,\,\,\,\,\hspace{2cm}\end{equation}  In particular, if
$u\not\equiv 0$, then $u$ is a weakly solution of (\ref{1.1}). Note
that the corresponding sum in (\ref{b1}) will be treated as zero if
$l_i=0\,(i=1,2).$
\end{theorem}
{\bf Remarks:}

 1)  Similar as Corollary 3.3 in \cite{Sm}, one can show that any Palais-Smale sequence for $I$ at a level
which is not of the form $m_1D+m_2 J^\infty$, 
$m_1,m_2\in\N\bigcup \{0\}$, gives rise to a non-trivial weak
solution of equation (1.1).



 2) In our non-compactness analysis, we prove
 that the blowing up positive Palais-Smale sequences can bear exactly
 two kinds of bubbles. Up to harmless constants, they are either
 of the form
 $$U^{R_n}(x), \text{
 }|R_n|\rightarrow 0\text{ as } n\rightarrow
 \infty,
$$or $$u(x-y_n)\in H^{1}(\R^N),\,\text{
}|y_n|\rightarrow\infty,\text{ as } n\rightarrow
 \infty,$$where $u$ is the solution of
(\ref{(2.1)}). For any Palais-Smale sequence ${u_n} $ for $I$, ruling out the above two bubbles yields  the existence of a non-trivial weak solution of equation (1.1).

Using above compact results and the Mountain Pass Theorem \cite{AR} we prove
the following corollary.
\begin{corollary}\label{b}
Assume that $2<q<2^*$ for $N\geq 4\alpha$. 
 If $a(x),k(x)$ satisfy (a) and (b),  and
\begin{equation}
\bar a\geq a(x), k(x)\geq \bar k>0,\,k(x)\not\equiv \bar k.
\end{equation}
 Then  (\ref{1.1}) has a nontrivial solution $u\in H^{\alpha}(\R^N)$ which
satisfies $$I(u)< \min\{\frac{2\alpha-s}{2(N-s)}S_{\alpha,s}^{\frac{N-s}{2\alpha-s}}, J^\infty\}.$$
\end{corollary}

This paper is organized as follows. In Section 2, we give some
preliminary lemmas.  In Section 3, we prove Theorem \ref{a} by
carefully analyzing  the features of a positive Palais-Smale
sequence for $I$. Corollary 1.1 is proved at the end of Section 3 by
applying Theorem \ref{a} and the Mountain Pass Theorem.

\section{Some preliminary lemmas}
In order to prove our main theorem, we give the following Lemmas.
\begin{lemma}\label{3.1}(Lemma 2.1, \cite{ZC}) 
   Let $\{\rho_n\}_{n\geq 1}$ be a sequence in $L^1(\R^N)$ satisfying
\begin{equation}\label{(2.4)}\rho_n\geq 0\,\,\, \text{on }
\R^N,\,\,\,\lim_{n\rightarrow\infty}\int_{\R^N}\rho_ndx=\la>0,
\end{equation} where $\la>0$ is fixed. Then there exists a
subsequence $\{\rho_{n_k}\}$ satisfying one of the  following two
possibilities:

(1) (Vanishing):
\begin{equation}\label{2.5}
\lim_{k\rightarrow\infty} \sup_{y\in\R^N}\int_{y+B_R}\rho_{n_k}(x)
dx=0,\,\,\,\, \text{for all } R<+\infty.
\end{equation}

(ii) (Nonvanishing): $\exists \alpha>0, R<+\infty$ and
$\{y_k\}\subset \R^N$ such that $$\displaystyle\lim_{k\rightarrow
+\infty}\int_{y_k+B_R}\rho_{n_k}(x)dx\geq \alpha >0.$$
\end{lemma}

\begin{lemma}\label{3.2}(Lemma 2.3, \cite{DS})
  Let $1\leq p<\infty$, with $p\neq \frac{2N}{N-2\alpha}$. Assume that $u_n$ is bounded in $L^p(\R^N)$, $(-\Delta)^{\alpha/2}u$
is bounded in $L^2(\R^N)$ and
$$\sup_{y\in\R^N}\int_{y+B_R}|u_n|^p dx\rightarrow 0
\,\,\text{for some}\,\,R>0\,\,\text{as}\,\,n\rightarrow \infty.$$
Then $u_n\rightarrow 0$ in $L^q (\R^N)$, for $q$ between
$p$ and $\frac{2N}{N-2\alpha}$.
\end{lemma}
\begin{lemma}\label{jl33}
Suppose that $0<s<2\alpha$ and $N>2\alpha$. Then there exists $C>0$ such that for any $u\in \dot{H}^\alpha(\R^N)$,
\begin{equation}\label{j2.3}
\bigl(\int_{\R^N}\frac{|u(x)|^{2^*_{s}}}{|x|^s}dx\bigl)^{\frac{2}{2^*_{s}}}\leq C\|u\|^2_{\dot{H}^\alpha(\R^N)}，
\end{equation}
a.e., 
\[\dot{H}^\alpha(\R^N)\hookrightarrow L^{2^*_{s}}_{}(\R^N,|x|^{-s})\]
is continuous. In addition, the inclusion 
\[\dot{H}^\alpha(\R^N)\hookrightarrow L_{loc}^p(\R^N,|x|^{-s}),\] 
is compact if $2\leq p<2^*_{s}$.
\end{lemma}
\begin{proof}
The proof of (\ref{j2.3}) is similar to that of Lemma 3.1 in \cite{Y}. Now we prove the compact impeding if $2\leq p<2^*_{s}$.
	Let  $\{u_n\}$ be a bounded sequence in $\dot{H}^\alpha(\R^N)$, then up to a subsequence (still denoted by $\{u_n\}$), 
	\[u_n \rightharpoonup u\text{ in } \dot{H}(\R^N).\]	
	Denote $v_n=u_n-u$, then for any $B(0,r)$, 
	\[v_n \rightharpoonup 0\text{ in } \dot{H}(\R^N),\ \ v_n \rightarrow 0 \text{ in } L^q(B(0,r)),\, 2\leq q <  2^*=\frac{2N}{N-2\alpha}.\]
	Fix $r>0$,
	since  $(p-\frac{s}{\alpha})(\frac{2\alpha}{2\alpha-s})<2^*$, it follows
	\begin{equation}\label{l2.41}
	\begin{split}
	\int_{B(0,r)}\frac{|v_n|^{p}}{|x|^s}dx&
=\int_{B(0,r)}\frac{|v_n|^{s/\alpha}}{|x|^s}|v_n|^{p-s/\alpha}dx\\
&\leq \Bigl(\int_{B(0,r)}（\frac{|v_n|^{2}}{|x|^{2\alpha}}dx\Bigl )^{\frac{s}{2\alpha}}
\Bigl (\int_{B(0,r)}|v_n|^{(p-\frac{s}{\alpha})(\frac{2\alpha}{2\alpha-s})}dx\Bigl)^{\frac{2\alpha-s}{2\alpha}}
\\&\leq c\|(-\Delta )^{\alpha/2}v_n\|_{L^2(\R^N)}^{\frac{s}{2\alpha}}\Bigl (\int_{B(0,r)}|v_n|^{(p-\frac{s}{\alpha})(\frac{2\alpha}{2\alpha-s})}dx\Bigl)^{\frac{2\alpha-s}{2\alpha}}\rightarrow 0,
	\end{split}
	\end{equation}
	and 
		\begin{equation}\label{l2.42}
		\begin{split}
		\int_{B(0,r)^C}\frac{|v_n|^{p}}{|x|^s}dx
		\leq\int_{B(0,r)^C}\frac{|v_n|^{p}}{r^s}dx\leq \frac{1}{r^s}\|v_n\|_{L^p(\R^N)}^p.
		\end{split}
		\end{equation}
		Letting $r\rightarrow \infty$, collecting (\ref{l2.41}) and (\ref{l2.42}), it implies
		that	
	\[u_n\rightarrow u \text{ in } L_{loc}^{p}(\R^N,|x|^{-s}).\]
This completes the proof.
\end{proof}

\begin{lemma}\label{3.5}
Let $\{u_n\}$ be a Palais-Smale  sequence of $I$ at level
$d\in\mathbb{R}$. Then $d\geq 0$ and $\{u_n\}$ $\subset$ $H^{\alpha}
(\R^N)$ is bounded. Moreover,
every Palais-Smale sequence for $I$ at a level zero converges
strongly to zero.
\end{lemma}
\begin{proof}
Since $a(x)\geq 0$, $\bar a>0$,
$\inf_{\R^N} a(x)=\hat{a}>0$, we have
$$\|u_n\|_{\dot{H}^\alpha (\R^N)}^2+\int_{\R^N}
a(x)|u_n|^2dx\geq
c\|u_n\|_{H^{\alpha}(\R^N)}^2,$$ and hence
for $q\leq 2^*_{s}$
\begin{eqnarray}\label{l3.1}
\begin{split}
 d+1+o(\|u_n\|) \geq &
I(u_n)-\frac{1}{q}\langle I'(u_n), u_n\rangle\\
=&(\frac{1}{2}-\frac{1}{q})\int_{\R^N}\bigl (|(-\Delta)^{\alpha/2}u_n|^2+a(x)|u_n|^2\bigl )dx \\&+(\frac{1}{q}-\frac{1}{{2^*_{s}}})\int_{\R^N}\frac{|u_n|^{{2^*_{s}}}}{|x|^s} dx\\
\geq & C\|u_n\|_{H^{\alpha}(\R^N)}^2 ,
\end{split}
\end{eqnarray}
for $2^*_{s}<q<2^*$,
 \begin{equation}\label{l3.2}
 \begin{split}
 d+1+o(\|u_n\|)&\geq
 I(u_n)-\frac{1}{2^*_{s}}\langle I'(u_n), u_n\rangle\\
 &=(\frac{1}{2}-\frac{1}{2^*_{s}})\int_{\R^N}\Bigl(|(-\Delta)^{\alpha/2}
 u_n|^2+a(x)|u_n|^2\Bigl)dx \\ &\ \ \ +(\frac{1}{{2^*_{s}}}-\frac{1}{q})\int_{\R^N}|u_n|^{q} dx\\
 &\geq C\|u_n\|_{H^{\alpha}(\R^N)}^2.
 \end{split}
 \end{equation}
 It follows from (\ref{l3.1}) and (\ref{l3.2})  that $\{u_n\}$ is bounded in $H^{\alpha}(\R^N)$ for $2<q<2^*$. Since
$$d=\lim_{n\rightarrow
\infty}I(u_n)-\max\{\frac{1}{q},\frac{1}{2^*_{s}}\}\langle I'(u_n),u_n\rangle\geq
C\limsup_{n\rightarrow \infty}\|u_n\|_{H^{\alpha}(\R^N)}^2,$$ we have
$d\geq 0$. Suppose now that $d=0$, we obtain from the above inequality
that
$$\lim_{n\rightarrow \infty}\|u_n\|_{H^{\alpha}(\R^N)} =0.$$
\end{proof}
Let $\{u_n\}$ be a Palais-Smale sequence. Up to a subsequence, we
assume that
$$u_n\rightharpoonup u_0 \,\,\,\text{ in } H^{\alpha}(\R^N)\text{ as
}n\rightarrow \infty.$$ Obviously, we have $I'(u_0)=0$.
 Let $v_n=u_n-u_0$, from Lemma \ref{jl33} as $n\rightarrow \infty$, 
 \begin{eqnarray}
v_n\rightharpoonup 0\text{ in
}H^{\alpha}(\R^N),\ \ \ \ \ \ \ \  \label{jl2.7}\\
v_n\rightarrow
0\text{ in
}L_{loc}^{p}(\R^N, |x|^{-s})\text{ for all } 2\leq p<2^*_{s}, \\
v_n\rightarrow
0\text{ in
}L^{q}_{loc}(\R^N)\text{ for all } 2< q<2^*.
 \end{eqnarray}
As a consequence, we have the following Lemma.
\begin{lemma}\label{q}
 $\{v_n\}$ is a Palais-Smale  sequence for $I$ at level
$d_0=d- I(u_0)$. \end{lemma}
\begin{proof}
 By the Br$\acute{e}$zis-Lieb Lemma in \cite{BN} and $v_n\rightharpoonup 0$ in $H^{\alpha}(\R^N)$ , as
 $n\rightarrow\infty$, we have
\begin{equation}\int_{\R^N} |v_n|^{q}dx=\int_{\R^N}
|u_n|^{q}dx-\int_{\R^N}|u_0|^{q}dx+o(1)\,\, \text{for all }\  2\leq
q\leq {2^*_{s}} ; \end{equation} 
\begin{equation}\int_{\R^N }\frac{|v_n|^{2^*_{s}}}{|x|^s}dx=\int_{\R^N}
\frac{|u_n|^{2^*_{s}}}{|x|^s}dx-\int_{\R^N} \frac{| u_0|^{2^*_{s}}}{|x|^s}dx+o(1); \,
\end{equation}
\begin{eqnarray}
\begin{split}
&\int_{\R^N}\int_{ \R^N}\frac{|u_{n}(x)-u_{n}(y)|^2}{|x-y|^{N+2\alpha}}dxdy\\=&	\int_{\R^N}\int_{ \R^N}\frac{|(v_{n}(x)+u(x))-(v_{n}(y)+u(y))|^2}{|x-y|^{N+2\alpha}}dxdy\\
=&\int_{\R^N}\int_{ \R^N}\frac{|v_{n}(x)-v_{n}(y)|^2+|u(x)-u(y)|^2+2(v_n(x)-v_n(y))(u(x)-u(y))}{|x-y|^{N+2\alpha}}dxdy\\
=&\int_{\R^N}\int_{ \R^N}\frac{|v_{n}(x)-v_{n}(y)|^2}{|x-y|^{N+2\alpha}}dxdy+\int_{\R^N}\int_{ \R^N}\frac{|u(x)-u(y)|^2}{|x-y|^{N+2\alpha}}dxdy+o(1).
\end{split}
\end{eqnarray} 
 Hence
$I(v_n)=I(u_n)-I(u_0)+o(1)=d-I(u_0)+o(1)$.

For $v\in C^\infty_0(\R^N)$, there exists a $B(0,r)$ such that $\dis \mathrm{supp}v \subset B(0,r)$. Then$\text{ as } n\rightarrow \infty$,
\begin{equation}
\begin{aligned}\label{l5.16}
|\int_{\R^N}k(x)|v_n|^{q-2}v_n vdx|\leq c
|\int_{B(0,r)}|v_n|^{q-2}v_n vdx|=o(1),
\end{aligned}
\end{equation}
and from Lemma \ref{jl33}, 
\begin{equation}\label{l5.17}
|\int_{\R^N }\frac{|v_n|^{2^*_{s}-2} v_n v}{|x|^s}dx|
\leq |\int_{|x|\leq r }\frac{|v_n|^{2^*_{s}-2} v_n v}{|x|^s}dx|
\leq c \int_{|x|\leq r }\frac{|v_n|^{2^*_{s}-1}}{|x|^s}dx=o(1).
\end{equation}

 By (\ref{jl2.7}), (\ref{l5.16}) and (\ref{l5.17}), we have $\langle v,
I'(v_n)\rangle=o(1)\text{ as }n\rightarrow
 \infty$.
\end{proof}
\begin{lemma}\label{3.7}
Let $\{v_n\}\subset H^{\alpha}(\R^N)$ be a Palais-Smale  sequence of $I$ at
level d and $v_n\rightharpoonup 0\text{ in } H^{\alpha}(\R^N)\text{
as }n\rightarrow \infty$. If there exists a sequence $\{r_n\}\subset
\R^+, \text{ with } r_n\rightarrow 0$ as $n\rightarrow \infty$ such
that $\bar v_n(x):=r_n^{\frac{N-2\alpha}{2}}v_n (r_n x)$ converges weakly
in $\dot{H}^\alpha(\R^N)$ and almost everywhere to some $0\neq v_0\in
\dot{H}^\alpha(\R^N)\text{ as }n\rightarrow \infty$, then $v_0$ solves
(\ref{1.4}) and the sequence
$z_n:=v_n-v_0(\frac{x}{r_n})r_n^{\frac{2\alpha-N}{2}}$ is a
Palais-Smale sequence of $I$ at level $d-I_s(v_0)$. \end{lemma}
\begin{proof}
First, we prove that $v_0$ solves (\ref{1.4}) and 
$I(z_n)=I(v_n)-I_\mu(v_0)$. Fix a ball $B(0,r)$ and a test
function $\phi\in C^\infty_0(B(0,r))$.
 Since $$\bar v_n(x)\rightharpoonup v_0 \text{  in } \dot{H}^\alpha(\R^N),$$ 
applying Lemma \ref{jl33}, it implies 
\begin{equation}\label{j62.21}
\begin{split}
\langle\phi, I'_s(v_0)\rangle&=\int_{\R^N}\int_{\R^N}\frac{(v_0(x)-v_0(y))(\phi(x)-\phi(y))}{|x-y|^{N+2\alpha}}
dxdy-\int_{\R^N}\frac{|v_0|^{{2^*_{s}}-2}v_0\phi}{|x|^s}
dx\\
&=\int_{\R^N}\int_{\R^N}\frac{(\bar v_n(x)-\bar v_n(y))(\phi(x)-\phi(y))}{|x-y|^{N+2\alpha}}
dxdy-\int_{\R^N}\frac{|\bar v_n|^{{2^*_{s}}-2}\bar v_n\phi}{|x|^s}
dx\\
& \ \ \ +{r_n}^{2\alpha}\int_{\R^N} a({r_n x})\phi\bar v_n
dx-r_n^{N-\frac{N-2\alpha}{2}q}\int_{\R^N}k(r_n x)\phi|\bar v_n|^{q-2}\bar v_ndx+o(1)\\
&=\int_{\R^N}\int_{\R^N}\frac{(v_n(x)-v_n(y))(\phi_n(x)-\phi_n(y))}{|x-y|^{N+2\alpha}}
dxdy-\int_{\R^N}\frac{|v_n|^{{2^*_{s}}-2}v_n\phi_n}{|x|^s}
dx\\
& \ \ \ +\int_{\R^N} a(x)\phi_n v_ndx -\int_{\R^N}
k(x)\phi_n|v_n|^{q-2}v_n dx+o(1) =o(1)\text{ as }n\rightarrow
\infty,
\end{split}
\end{equation}
where $\phi_n(x)=r_n^{\frac{2\alpha-N}{2}}\phi(x /r_n)$. The last equality in  (\ref{j62.21})  holds since
$\|\phi\|_{\dot{H}^\alpha(\R^N)}=\|\phi_n\|_{H^{\alpha}(\R^N)}+o(1)\text{ as
}n\rightarrow \infty$. Thus $v_0$ solves (\ref{1.4}).
 Let
$$ z_n(x)=v_n(x)-r_n^{\frac{2\alpha-N}{2}}v_0(\frac{x}{r_n})\in
H^{\alpha}(\R^N),$$ then $$\bar z_n=r_n^{\frac{N-2\alpha}{2}}z_n(r_n x)=\bar
v_n-v_0(x).$$ Obviously $z_n\rightharpoonup 0 \text{
in }H^{\alpha}(\R^N)\text{ as }n\rightarrow \infty$. Now we prove
that $\{z_n\}$ is a Palais-Smale  sequence of $I$ at level $d-I_s
(v_0)$.

Since
\begin{equation}\label{x}
\begin{split}
\int_{\R^N}|v_0(\frac{x}{r_n})r_n^{\frac{2\alpha-N}{2}}|^pdx=r_n^{N-p\frac{(N-2\alpha)}{2}}\|v_0\|^p_{L^p(\R^N)}\rightarrow 0 ,\,\,\text{ as }n\rightarrow\infty,\text{ for all }{1\leq p< {2^*_{\alpha}}},
 \end{split}
 \end{equation}
 by the Br$\acute{e}$zis-Lieb Lemma and the weak convergence, similar to Lemma \ref{q}, we can prove
 have
$$I(z_n)=I(v_n)-I_s(v_0),$$
$$ \langle I'(z_n),\phi\rangle=o(1)$$ as $n\rightarrow \infty$.
It completes the proof.
\end{proof}
\begin{lemma}\label{l6}
	Let $0<\alpha <N/2,\, 0<s<2\alpha $,  $\{u_n\}\subset {\dot{H}}^\alpha(\R^N)$ be a bounded sequence such that
	\begin{equation}\label{l6.1}
\inf_{n\in N}\int_{\R^N}\frac{|u_n|^{2^*_{s}}}{|x|^s}dx\geq c>0.
	\end{equation}
Then, up to subsequence, there exist a family of positive numbers $\{r_n\}\subset \R^N$ such that 
\begin{equation}\label{l6.2}
\bar u_n \rightharpoonup w\neq 0 \text{ in } \dot{H}^\alpha(\R^N), \end{equation}
where $\bar u_n=r_n^{\frac{N-2\alpha}{2}}u_n(r_n x)$.
\end{lemma}
\begin{proof}
For the proof of (\ref{l6.2}), refer to the proof of Theorem 1.3 in \cite{Y1}. Here we Omit it.
\end{proof}
\section{ Non-compactness analysis}
In this section, we prove Theorem 1.1 by Concentration-Compactness
Principle and a delicate analysis of the Palais-Smale sequences of
$I$.

 {\bf Proof of Theorem \ref{1.1}.}
 By Lemma \ref{3.5}, we can assume that $\{u_n\}$ is
bounded. Up to a subsequence, $\text{ let }n\rightarrow \infty$, we
assume that
\begin{gather}
\label{t1.11}u_n\rightharpoonup u \text{ in } H^{\alpha}(\R^N),\\
\label{t1.12}u_n\rightarrow u \text{ in } L^p_{loc}(\R^N)\,\text{ for }1<
p<{2^*_{\alpha}},\\
\label{t1.13}u_n\rightarrow u \text{ a.e. in } \R^N.
\end{gather}
Denote $v_n=u_n-u$, then $\{v_n\}$ is
a Palais-Smale  sequence of $I$ and $v_n\rightharpoonup0$  in
$H^{\alpha}(\R^N)$. Then by Lemma \ref{q} we know that
\begin{gather}
\label{4.1}I(v_n)=I(u_n)-I(u)+o(1),\text{ as }n\rightarrow \infty,\\
\label{4.2}I'(v_n)=o(1),\text{ as }n\rightarrow \infty,\\
\label{4.3}\|v_n\|_{H^{\alpha}(\R^N)}=\|u_n\|_{H^{\alpha}(\R^N)}-\|u\|_{H^{\alpha}(\R^N)}+o(1),\text{
as }n\rightarrow \infty.
\end{gather}

Without loss of generality, we may assume that
$$\|v_n\|^2_{H^{\alpha}(\R^N)}\rightarrow l>0\text{ as
}n\rightarrow\infty.$$ In fact if $l=0$, Theorem 1.1 is proved for $l_1=0,l_2=0$.

{\bf Step 1:}\, Getting rid of the blowing up bubbles caused by the Hardy
term.

Suppose there exists $0<\delta<\infty$ such that
 \begin{equation}\label{4.6}
\int_{|x|<R}\frac{|v_n|^{2^*_{s}}}{|x|^s} dx\geq \delta>0,\,\,\text{ for
some } \,0<R<\infty.
\end{equation}

It follows from Lemma \ref{l6} that there exist a positive sequence $\{r_n\}\subset \R$ such that $$\bar v_n=\dis r_n^{\frac{N-2\alpha}{2}}v_n(r_nx)\rightharpoonup v_0\not\equiv 0\text{ in } \dot{H}^{\alpha}(\R^N)$$
Now we claim that $r_n\rightarrow 0 \text{ as } 	n\rightarrow \infty.$
In fact there exist $R_1>0$ such that 
\begin{equation}\label{jt3.8}
	\int_{B(0,R_1)}|v_0|^pdx=\delta_1>0, \text{ for }1<p<2^*_{\alpha}.
\end{equation}
From the Sobolev compact embedding and (\ref{t1.11})-(\ref{t1.12}), we have that
\[ v_n\rightarrow 0\text{ in } L^p(B(0,r)) \text{ for all }1<p<2^*_{\alpha},\]
\[ \bar v_n\rightarrow v_0
\text{ in } L^p(B(0,r)) \text{ for all }1<p<2^*_{\alpha},\]
\begin{equation}\label{t1.18}
\begin{split}
0\neq\|v_0\|_{L^2(B(0,r))}^2+o(1)=\int_{B(0,r)}|\bar v_n|^2dx=r_n^{-2\alpha}\int_{B(0,r_nR)}|v_n|^2dx.
\end{split}
\end{equation}
If $r_n\rightarrow r_0>0$, then $$r_n^{-2\alpha}\int_{B(0,r_nR_1)}|v_n|^2dx \leq c r_0^{-2\alpha}\|v_n\|^2_{L^2(B(0,cR_1))}\rightarrow 0;$$
if $r_n\rightarrow \infty$, then $$r_n^{-2\alpha}\int_{B(0,r_nR_1)}|v_n|^2dx\leq r_n^{-2\alpha}\|v_n\|^2_{H^\alpha(\R^N)}\rightarrow  0.$$ A contradiction to (\ref{t1.18}). Thus we have $r_n\rightarrow 0$.

 Define $z_n =v_n-v_0(\frac{x}{r_n})r_n^{\frac{2\alpha-N}{2}}\varphi$, then $ z_n \rightharpoonup 0 $ in $H^\alpha (\R^N)$. It follows from Lemma
 \ref{3.7} that $\{z_n\}$ is a Palais-Smale sequence of $I$
 satisfying
 \begin{gather}
 \label{4.12}I(z_n)=I(v_n)-I_s(v_0)+o(1), \text { as } n\rightarrow\infty. 
 \end{gather}
If still there exists a $\delta>0, \text{ such that }
\int_{|x|<R}\frac{|z_n|^{2^*_{s}}}{|x|^s}dx\geq \delta
 >0$, then repeat the previous argument. The iteration must stop after
 finite times. And we will have a new Palais-Smale sequence of
 $I$,
 (without loss of generality) denoted by $\{v_n\}$, such that
 \begin{equation}\label{4.15}\int_{|x|<R}\frac{|v_n|^{2^*_{s}}}{|x|^s} dx=o(1),
 \text{ as } n\rightarrow \infty,\,\,\,\text{ for any } 0<R<\infty,
 \end{equation}
 and $v_n\rightharpoonup 0\text{ in } H^{\alpha}(\R^N) \text{ as } n\rightarrow \infty$.

{\bf Step 2}: Getting rid of the blowing up bubbles caused
by unbounded domains.

Suppose there exists $0<\delta <\infty$ such that
\begin{equation}\label{4.5}
\|v_n\|_{H^\alpha(\R^N)}^2\geq c(\int_{\R^N}|v_n|^qdx)^{\frac{2}{q}}\geq \delta >0.
\end{equation}
 By Lemma \ref{3.1},
 there exists a subsequence still denoted by $\{v_n\}$, such that one
 of the following two cases occurs.
 
 i) Vanish occurs.
 $$\displaystyle\forall \,0< R<\infty, \sup_{y\in\R^N}\int_{y+B_R}(|（(-\Delta ) ^{\alpha/2}v_n|^2+|v_n|^2)dx\rightarrow 0\text{ as }n\rightarrow \infty.$$
 By the Sobolev inequality and
 Lemma \ref{3.2} we have $$\int_{\R^N}|v_n|^pdx\rightarrow 0\text{ as
 }n\rightarrow \infty,\,\,\forall\,1<p<2^*_{s},$$ which contradicts
 (\ref{4.5}).
 
 ii) Nonvanish occurs.
 
 There exist $\beta
 >0,\,0<\bar R<\infty,\,\{y_n\}\subset\R^N \text{ such that }$
 \begin{equation}\label{4.4}
 \liminf_{n\rightarrow \infty}\int_{y_n+B_{\bar R}}(|(-\Delta)^{\alpha/2}
 v_n|^2+|v_n|^2)dx\geq \beta
 >0.
 \end{equation}
 
 We claim $|y_n|\rightarrow\infty$ as
 $n\rightarrow\infty$. Otherwise, $\{v_n\}$ is tight, and thus
 $\|v_n\|_{L^q(\R^N)}\rightarrow 0$ as $n\rightarrow \infty$. This
 contradicts (\ref{4.5}).
 
Fo proceed, we first construct the Palais-Smale sequences of $I^\infty$.

  Denote $\bar
 v_n=v_n(x+y_n)$.
 Since $\|\bar v_n\|_{H^{\alpha}(\R^N)}=\|v_n\|_{H^{\alpha}(\R^N)}\leq c$, without
 loss of generality, we assume that $\text{ as } n\rightarrow
 \infty$,
 $$\bar v_n\rightharpoonup v_0\text{ in }H^{\alpha}(\R^N),$$
 $$\bar v_n\rightarrow v_0 \text{ in } L^p_{loc}(\R^N),\text{ for any }
 1< p < {2^*_{\alpha}}.$$
 By (\ref{4.15}), we have $\forall \phi\in C^\infty_0(\R^N)$,
 \begin{align*}\int_{\R^N}\frac{|
 v_n|^{2^*_{s}-2} v_n\phi}{|x+y_n|^s}dx&=\int_{\R^N}\frac{|
 v_n|^{2^*_{s}-2} v_n\phi_n}{|x|^s}dx\\&=\int_{|x|>r}\frac{|v_n|^{2^*_{s}-2}v_n\phi_n}{|x|^s}dx+o(1)
 \\&\leq\frac{1}{r^s}(\int_{\R^N}|v_n|^{\frac{(2^*_{s}-1)2^*}{2^*-1}}dx)^{\frac{2^*-1}{2}}(\int_{\R^N}|\phi_n|^{2^*}dx)^{1/2^*}+o(1)\text{ as } n\rightarrow
 \infty,
 \end{align*} where $\phi_n=\phi(x-y_n)$.
 Let $r\rightarrow \infty$, since $\frac{(2^*_{s}-1)2^*}{2^*-1}<2^*$, we have
 \begin{equation}\label{4.16}\int_{\R^N}\frac{ |\bar v_n|^{2^*_{s}-2}\bar v_n \phi}{|x+y_n|^s}dx=o(1)\text{ as }
 n\rightarrow \infty.
 \end{equation}
 Similarly we have \begin{equation}\label{4.17}\int_{\R^N}\frac{|\bar v_n|^{2^*_{s}}}{|x+y_n|^s}dx=o(1)\text{ as }
 n\rightarrow \infty.
 \end{equation}
 Since $v_n\rightharpoonup 0 \text{ in
 }H^{\alpha}(\R^N)$ and $\dis\lim_{n\rightarrow\infty}a(x+y_n)=\bar a$,
 we have as $n\rightarrow\infty$,
 $$
o(1)= \int_{\R^N}a(x)v_n\phi_n dx=\int_{\R^N}\bar a \bar v_n\phi
dx+\int_{\R^N}[a(x+y_n)-\bar a] \bar v_n\phi dx $$
 and $$|\int_{\R^N}[a(x+y_n)-\bar a] \bar v_n\phi dx|\leq c(\int_{\R^N}|a(x+y_n)-\bar
 a
 |^2\phi^2dx)^{1/2}=o(1),$$
 that is,
 \begin{equation}\label{s3}
\int_{\R^N}\bar a \bar v_n\phi dx=o(1)=\int_{\R^N}a(x)v_n\phi_n
dx\text{ as } n\rightarrow \infty.
\end{equation}
Similarly we have
 \begin{equation}\label{s4}
\int_{\R^N} k(x) | v_n|^{q-2} v_n\phi_n dx=\int_{\R^N}\bar k |\bar
v_n|^{q-2}\bar v_n\phi dx=o(1)\text{ as } n\rightarrow \infty.
\end{equation}
Recall that $v_n$ is a Palais-Smale
 sequence of $I$, by (\ref{4.16})-(\ref{s4}) we have
 \begin{equation}\langle I'(v_n),\phi_n\rangle+o(1)=\langle {I^\infty}' (\bar
 v_n),\phi\rangle=o(1),\text{ as
}n\rightarrow \infty.\end{equation}
 This shows that $\bar v_n$ is a nonnegative Palais-Smale
 sequence of $I^\infty(u)$, and $v_0$ is a weak solution of
(\ref{(2.1)}).

 We claim that $v_0\not\equiv 0$.
From (\ref{4.5}), we may assume there exists a sequence $\{y_n\}$
satisfying (\ref{4.4}) and
\begin{equation}\label{z3}\int_{B(y_n,R)}|v_n|^qdx=b+o(1)>0,\text{ as
}n\rightarrow \infty,
\end{equation} where $b>0$ is a constant.

If $v_0=0$, we have $\int_{B(R)}|\bar v_n|^qdx=\int_{B(y_n,R)}|
v_n|^qdx=o(1)\text{ as } n\rightarrow \infty$ for $0<R<\infty$ which
contradicts (\ref{z3}).


  Denote $z_n=\bar v_n-v_0$. Since \begin{align*}
 I(v_n)&=\frac{1}{2}\int_{\R^N}\Bigl (|(-\Delta)^{\alpha/2}v_n|^2+a(x)|v_n|^2
\Bigl )dx
-\frac{1}{{2^*_{s}}}\int_{\R^N}\frac{|v_n|^{{2^*_{s}}}}{|x|^s}dx-\frac{1}{q}\int_{\R^N}k(x)|v_n|^q
dx\\
&=\frac{1}{2}\int_{\R^N} \Bigl (|(-\Delta)^{\alpha/2}\bar v_n|^2+a(x+y_n)|\bar
v_n|^2 \Bigl )dx-\frac{1}{{2^*_{s}}}\int_{\R^N}\frac{|\bar
v_n|^{{2^*_{s}}}}{|x+y_n|^s}dx\\
&\ \ \ -\frac{1}{q}\int_{\R^N}k(x+y_n)|\bar v_n|^q dx\\
&=\frac{1}{2}\int_{\R^N}\Bigl (|(-\Delta)^{\alpha/2}\bar v_n|^2+\bar a|\bar
v_n|^2 \Bigl )dx-\frac{1}{q}\int_{\R^N}\bar k|\bar v_n|^q dx+o(1),
\end{align*}
where the last equality is a result of (\ref{4.17}), therefore, as $n\rightarrow \infty$,
\begin{gather}\label{4.20}\|z_n\|_{H^{\alpha}(\R^N)}=\|\bar
v_n\|^2_{H^{\alpha}(\R^N)}-\|v_0\|^2_{H^{\alpha}(\R^N)}+o(1), \\
\label{4.21} I(z_n)=I^\infty(\bar v_n)-I^\infty
(v_0)+o(1)=I(v_n)-I^\infty(v_0)+o(1).
\end{gather}

Hence
$z_n\rightharpoonup0\text{ in }H^{\alpha}(\R^N)\text{ as } n\rightarrow
\infty$,
 and $z_n$ is a Palais-Smale  sequences of
$I$. If $\|z_n\|_{L^q(\R^N)}\rightarrow c>0\text{ as
}n\rightarrow\infty$, then one can repeat Step 2 for finite times, since the amount of
sequences satisfing (\ref{4.4}) is finite.

Thus we  obtain a new Palais-Smale sequence of $I$, without loss of generality still denoted by $v_n$, such that
$$\|v_n\|_{L^q(\R^N)}\rightarrow 0,\ \ \int_{\R^N}\frac{|v_n|^{2^*_{s}}}{|x|^s}dx\rightarrow 0$$
as $n\rightarrow\infty$. Then we have 
$$\|v_n\|^2_{H^\alpha(\R^N)}\leq c\int_{\R^N}(| \bigl (-\Delta )^{\alpha/2}v_n|^2+a(x)|v_n|^2\bigl )dx\leq c\bigl(\|v_n\|^q_{L^q(\R^N)}+ \int_{\R^N}\frac{|v_n|^{2^*_{s}}}{|x|^s}dx\bigl ) \rightarrow 0$$ as $n\rightarrow \infty$.
The proof of Theorem 1.1 is complete.

Now we are ready to prove corollary 1.1 by Mountain Pass Theorem and
Theorem 1.1.

 {\bf Proof of Corollary \ref{b}:}
From
\begin{align*}
I (tu)=&\frac{t^2}{2}\Bigl[\int_{\R^N}|(-\Delta)^{\alpha/2}
u|^2dx+\int_{\R^N}a(x)u^2
dx\Bigl]\\
&-\frac{|t|^{{2^*_{s}}}}{{2^*_{s}}}\int_{\R^N}\frac{|u|^{{2^*_{s}}}}{|x|^s}dx-\frac{
|t|^q}{q}\int_{\R^N}k(x)|u|^q dx,
\end{align*}
we deduce that for a fixed $u\not\equiv 0$ in $H^{\alpha}(\R^N)$, $I
(tu)\rightarrow -\infty$ if $t\rightarrow \infty$. Since
 $$
\int_{\R^N}|u|^q dx\leq
C\|u\|^q_{H^{\alpha}({\R^N})},\,\,\,\int_{\R^N}\frac{|u|^{{2^*_{s}}}}{|x|^s} dx\leq
C\|u\|^{{2^*_{s}}}_{H^{\alpha}({\R^N})},
 $$
 we have
$$
I(u)\geq
c\|u\|_{H^{\alpha}(\R^N)}^2-C(\|u\|_{H^{\alpha}(\R^N)}^q+\|u\|_{H^{\alpha}(\R^N)}^{{2^*_{s}}}).
$$
Hence, there exists $r_0>0$  small such that $I(u)\Bigl|_{\pa
B(0,r_0)}\geq\rho>0$ for $q,\,{2^*_{s}}>2$.

As a consequence,  $I(u)$ satisfies the geometry structure of
Mountain-Pass Theorem. Now define
$$ c^*=:\inf_{\gamma \in \Gamma} \sup_{t
\in [0,1]}I(\gamma (t)),$$ where $ \Gamma=\{\gamma \in
C([0,1],H^{\alpha}(\R^N)):\gamma(0)=0,\gamma(1)=\psi_0\in H^{\alpha}(\R^N)\}$ with $I(t\psi_0)\leq 0$ for all $t\geq 1$.

To complete the proof of Corollary 1.1, we need to verify that
$I(u)$ satisfies the local Palais-Smale conditions. According to
Remarks 2),
 we only need to verify that
\begin{equation}\label{4.28}
c^*<\min\{\frac{2\alpha-s}{2(N-s)}S_{\alpha,s}^{\frac{N-s}{2\alpha-s}},
J^\infty\}.
\end{equation}  

 Set $v_\va(x)=\frac{U_\va}{(\int_{\R^N}\frac{|U_\va|^{{2^*_{s}}}}{|x|^s}dx)^{1/{2^*_{s}}}}$. We
 claim
\begin{equation}\label{4.29}
\max_{t>0}I(tv_\va)<\frac{2\alpha-s}{2(N-s)}S_{\alpha,s}^{\frac{N-s}{2\alpha-s}}.
\end{equation}  In
fact,  from (\ref{1.5}) it is easy to calculate  the following estimates
\begin{eqnarray}
\|v_\va\|_{\dot{H}^\alpha(\R^N)}=S_{\alpha,s}, \label{jt3.19}
\end{eqnarray}
\begin{eqnarray}
\dis
\int_{\R^N}|v_\va|^2dx\leq c\va^{2\alpha-N}\int_{\R^N}\frac{1}{(1+|\frac{x}{\va}|^2)^{N-2\alpha}}dx\leq 
\begin{cases}O(\va^{2\alpha}),\,\,\,\,&N\geq 4\alpha,\\
O(\va^{2\alpha}|\log\va|),&N=4\alpha; \end{cases}\label{jt3.20}
\end{eqnarray}
\begin{eqnarray}
\int_{\R^N}|v_\va|^qdx\geq O(\va^{\frac{(2\alpha-N)q}{2}+N}).\label{jt3.21}
\end{eqnarray}
 Since $2^*>q>2$ we
 have
 \begin{equation}\label{l}
  O(\va^{2\alpha})=o(\va^{\frac{(2\alpha-N)q}{2}+N}),\,\,O(\va^{2\alpha}|\log\va|)=o(\va^{\frac{(2\alpha-N)q}{2}+N}).
 \end{equation}
 Denote
$t_\va$ be the attaining point of $\dis\max_{t>0}I(tv_\va)$, we can
prove that $t_\va$ is uniformly bounded (see \cite{DG}). Hence, for
$\va>0$ sufficient small,
\begin{align*}
 \max_{t>0}I(tv_\va)&=I (t_\va v_\va)\\&\leq
 \max_{t>0}\Bigl\{\frac{t^2}{2}\int_{\R^N}\Bigl(|(-\Delta)^{\alpha/2}
v_\va|^2\Bigl
)dx-\frac{t^{{2^*_{s}}}}{{2^*_{s}}}\int_{\R^N}
\frac{|v_\va|^{{2^*_{s}}}}{|x|^s}dx\Bigl\}\\&\ \ \  -O(\va^{\frac{(2\alpha-N)q}{2}+N}) +
\begin{cases}O(\va^{2\alpha}),\,\,\,\,&N> 4\alpha,\\
O(\va^{2\alpha}|\log\va|),&N=4\alpha; \end{cases} \\
&<\frac{2\alpha-s}{2(N-s)}S_{\alpha,s}^{\frac{N-s}{2\alpha-s}}\text{ \,\,( by (\ref{l}) )}.
\end{align*}
This completes the proof of (\ref{4.29}). By the definition of
$c^*$, we have $c^*<\frac{2\alpha-s}{2(N-s)}S_{\alpha,s}^{\frac{N-s}{2\alpha-s}}$.

Next we verify
\begin{equation}\label{4.30}
c^*<J^\infty.
\end{equation}
Let $\{u_0\}$ be the minimizer of
$J^\infty,\,I^\infty(u_0)=J^\infty$ and
\[\int_{\R^N}\Bigl(|(-\Delta)^{\alpha/2} u_0|^2+\bar a
u_0^2
\Bigl)dx=\int_{\R^N}\bar
k|u_0|^q dx.\]
 Let
$$g(t)=I^\infty(tu_0)=\frac{1}{2}t^2\int_{\R^N}\Bigl(|(-\Delta)^{\alpha/2} u_0|^2+\bar a
u_0^2
\Bigl)dx-\frac{t^q}{q}\int_{\R^N}\bar
k|u_0|^q dx,$$
$$g'(t)=t\int_{\R^N}\Bigl(|(-\Delta)^{\alpha/2} u_0|^2+\bar a
u_0^2
\Bigl)dx-t^{q-1}\int_{\R^N}\bar
k|u_0|^q dx.$$ Thus $g'(t)\geq0$ if $t\in(0,1)$; $g'(t)\leq 0$ if
$t\geq 1$. Then
\begin{equation}\label{4.31}
\begin{aligned}
g(1)&=I^\infty(u_0)=\max_{l}I^\infty (u);\\
\text{ where }&l=\{tu_0,t\geq 0, u_0\text{ fixed }\}.
\end{aligned}
\end{equation}
Since there exists a $t_0>0$ such that $\displaystyle \sup_{t\geq
0}I(tu_0)=I(t_0u_0)$, from (\ref{4.31}) and the assumptions of
$a(x) \text{ and } k(x)$, we have
$$\displaystyle J^\infty=I^\infty (u_0)\geq
I^\infty (t_0u_0)>I(t_0 u_0)=\sup_{t\geq 0}I(t u_0).$$ It
  proves (\ref{4.30}). By (\ref{4.29}) and (\ref{4.30}) we have
(\ref{4.28}). Then the proof is completed.


\end{document}